\documentclass[12pt]{article}

\usepackage{hyperref}
\usepackage{amsfonts,amsmath,amssymb,amsthm, mathrsfs}
\usepackage{verbatim}
\usepackage{abstract}
\usepackage{anysize}%*

\marginsize{2.5cm}{2.5cm}{0.9cm}{1.9cm}

\newtheorem{rmk}{Remark}[section]

\numberwithin{equation}{section}

\newtheorem*{theorem*}{Theorem}

\newtheorem{theorem}{Theorem}[section]

\newtheorem{corollary}[theorem]{Corollary}

\marginsize{2.6cm}{2.6cm}{1cm}{2cm}
\def\p{\partial}

\def\i{\sqrt{-1}}

\def\o{\omega}

\def\cH{{\mathcal H}}
\def\cI{{\mathcal I}}

\def\tr{\mathop{\rm tr}\nolimits}

\title{Compactness of K\"ahler metrics with bounds on  Ricci curvature and $\mathcal I$ functional}
\author{Xiuxiong Chen, Tam\'as Darvas, Weiyong He}

\date{}

\begin{document}

\maketitle
\begin{abstract}We prove a compactness theorem for K\"ahler metrics with various bounds on Ricci curvature and the $\mathcal I$ functional. We explore applications of our result to the continuity method and the Calabi flow.
\end{abstract}

\section{Introduction}In this short note we discuss  compactness results for K\"ahler metrics in a fixed K\"ahler class, under geometric assumptions. Convergence and compactness of K\"ahler metrics has many special properties compared with the convergence of Riemannian metrics in general. 

We fix a compact K\"ahler manifold $X$, of complex dimension $n$, with a  K\"ahler class $[\omega]$.  Up to a constant, a K\"ahler metric $\tilde \omega\in [\omega]$
can be written in terms of a K\"ahler potential, i.e, $\tilde \o = \o_\phi:= \o + i\partial \bar \partial \phi$. Accordingly, we denote the space of K\"ahler potentials as
\[
\cH:=\{\phi\in C^{\infty}(X): \omega_\phi=\omega+\i \p\bar\p \phi>0\}.\]
The correspondence between metrics and potentials allows to phrase compactness theorems about K\"ahler metrics (in the same class) in terms of their (normalized) potentials.
%In particular, it makes perfect sense to consider the compactness of K\"ahler potentials in $\cH$, with the normalization assumption $\sup \phi=0$.  Equivalently, the compactness of K\"ahler potentials would imply the compactness of K\"aler metrics in holomorphic charts. 
A well known result of this type shows up in  Yau's proof of the Calabi conjecture \cite{y}. Stated in a technical way, it asserts that the  set of K\"ahler potentials  $\phi\in \cH$  with $C^2$-bounded log-volume ratio $\log\frac{\omega_\phi^n}{\omega^n}$  is  $C^{3, \alpha}$-compact for any $\alpha\in (0, 1)$ (i.e. the corresponding metrics are $C^{1, \alpha}$-compact). Using pluripotential theory, Kolodziej \cite{ko} proved that compactness of K\"ahler potentials holds under much weaker assumptions: if the volume ration is in $L^p, \ p>1$, then the set of potentials is $C^{\alpha}$-compact.

Using Kolodziej's estimates and integral methods, the first and the third author generalized Yau's result on the compactness of K\"ahler metrics and proved that if the log volume ratio is in $W^{1, p}, p>2n$, then the K\"ahler potenials are $C^{2, \alpha}$-compact \cite{ch2}. All these results were achieved by studying the complex Monge-Ampere equation. 

Another problem of interest is to consider the compactness of K\"ahler potentials under curvature conditions, in particular Ricci curvature bound, given the close relation of Ricci curvature and the volume form. Clearly, additional assumptions are needed to guarantee compactness. By using Yau's technique \cite{y}, the first  and third author proved that K\"ahler potentials are $C^{3, \alpha}$-compact given uniform bound on Ricci curvature and the K\"ahler potentials \cite[Theorem 5.1]{ch1}, leading to applications related to the Calabi flow: %A key observation in  is instead of the Ricci lower bound, the Ricci upper bound plays a very key role on the compactness of K\"ahler potentials. One main technical result in \cite{ch} can be stated as follows, 
%Tamas: I understand that the proof of the result tells a little more then this, but for simplicity and a better narrative, I would keep this  result more concise, following the original statement from {ch}.
\begin{theorem}[Chen-He]\label{tch} Consider the set of potentials $\phi \in \mathcal H$ with $\sup_X \phi =0$.
If both $Ric \ {\omega_\phi}$ and $\|\phi\|_{C^0}$ are uniformly bounded, then there exists a  $C >1$ such that
\[\frac{1}{C}\omega\leq \omega_\phi\leq C\omega  \ \textup{ and } \ \| \phi\|_{C^3,\alpha} \leq C,
\]
for all $\alpha \in (0,1)$.
\end{theorem}

The assumption on the $C^0$ bound is not satisfactory in various settings. For example, when considering equations that govern existence of canonical K\"ahler metrics, one is often initially led to energy bounds on the potentials instead of $C^0$ bounds.

To put the problem in a more geometric context, we recall the $L^2 $ Mabuchi metric and its $L^p$ generalizations defined on $\cH$: 
\[\|\psi\|_\phi^p=\int_M |\psi|^p \omega_\phi^n,\] 
where $\phi \in \mathcal H$ and $\psi \in T_\phi \mathcal H$ is a ``tangent vector". 
These metrics have been studied extensively, and are closely related to existence of canonical K\"ahler metrics, per Donaldson's program \cite{do}. In \cite{c1} the first author solved Donaldson's conjecture on the $L^2$ geodesic equation and confirmed that the Mabuchi path length pseudo distance $d_2$ is indeed a distance on $\mathcal H$. The $L^p$ analog of this same result was obtained in \cite{da}, i.e., it was shown that the $d_p$ path length pseudo distances are bona fide distances on $\mathcal H$ for all $p \geq 1$.

Along these lines, instead of assuming a $C^0$ bound, one can study the compactness of K\"ahler potentials with bounded Ricci curvature and bounded $L^2$ Mabuchi distance. This problem was proposed around 2006 by the first author. Unfortunately the estimates obtained on the Mabuchi distance at that time were not effective enough. Recently, in \cite{da} the second author has proved effective estimates comparing  the $d_p$ metrics to 	``more friendly" analytic expressions (see \eqref{doubleest} below). In particular %when considering potentials $\phi \in \mathcal H$ with vanishing Monge-Amp\`ere energy (sometimes denoted Aubin-Yau or Aubin-Mabuchi energy), 
the following estimate holds for all  $d_p$ metrics (see \cite{DR} for more precise estimates related to the $d_1$ metric):
\begin{equation}\label{eq: I_d_p_est}
\mathcal I(\o_u, \o_v):=\int_X (v-u) (\omega_u^n-\omega_v^n) \leq C d_p(u,v),  \ \ u,v \in \mathcal H,
\end{equation}
for some absolute constant $C:=C(n)>1$.

The main result of this short note is a compactness theorem for K\"ahler potentials with bounded Ricci curvature and bounded $\cI$ functional. In short, Theorem \ref{tch} still holds if we replace the $C^0$ bound by the bound on $\cI$-functional, with some precisions made along the way:
\begin{theorem}[Theorem \ref{main}, Corollary \ref{cor1}] For  $C>0$ consider the set of potentials $\phi \in \mathcal H$ with $\sup_X \phi =0$ and $\mathcal I(\omega, \omega_\phi) \leq C$.
\begin{itemize}
\item[(i)] If $Ric_{\omega_\phi} \leq C \o_\phi$ then there exists $D:=D(X,\omega,J,C)>1$ such that 
$$0 \leq \omega_\phi \leq D \omega.$$
In particular, $\|\phi\|_{C^{1,\alpha}} \leq D'$ for any $\alpha \in (0,1)$.
\item[(ii)] If $-C \o_\phi \leq Ric_{\omega_\phi} \leq C \o_\phi$
then 
$$\frac{1}{D} \omega \leq \omega_\phi \leq D \omega \ \textup{ and } \ \|\phi \|_{C^{3,\alpha}} \leq D.$$
\end{itemize}
\end{theorem}

By the discussion preceding the theorem, the same result holds if we replace the bound on $\mathcal I(\o,\o_\phi)$ with a corresponding bound on $d_p(0,\phi)$, for any $p \geq 1$.

As in \cite{ch1}, this last compactness result has direct applications to the smooth Calabi flow. Recall that by the results of this latter paper, uniform Ricci curvature bound along the Calabi flow implies existence of the flow for all times $t \geq 0$. Our theorem in this direction is the following:
\begin{theorem}\label{thm: Calabif} Suppose $t \to c_t , \ t \in [0,\infty)$ is a smooth Calabi flow trajectory with uniformly bounded Ricci curvature. Then the following hold:
\begin{itemize}
\item[(i)] If there exists a cscK metric in $\mathcal H$ then $t \to c_t$ converges smoothly to one such metric as $t \to \infty$.
\item[(ii)] If the Mabuchi K-energy is proper then there exists a cscK metric in $\mathcal H$ and $t \to c_t$ converges smoothly to one such metric.
\end{itemize}
\end{theorem}

Properness of the K-energy is understood as introduced by Tian (see \cite[Chapter 7]{t}). Part (i) in the above theorem strengthens a theorem of the third author \cite{h}, proved in the case when $X$ does not admit non-trivial holomorphic vector-fields. 
We mention that according to  recent work of Li-Wang-Zheng \cite{LWZ} if an extremal metric exists, and the Calabi flow exists for all time with bounded scalar curvature (such a curvature bound is not needed on K\"ahler surface), then the Calabi flow converges to an extremal metric, after taking a subsequence. 
We also note that $L^1$ convergence of the Calabi flow to a cscK metric holds unconditionally, as proved in \cite{bdl1}, but convergence in H\"older norms is not yet known. Part (ii) strengthens \cite[Theorem 2]{sz}, where a bound on the full curvature tensor is assumed instead.

Lastly, we consider the continuity path for the constant scalar curvature (cscK) equation,  via twisted cscK metrics, as introduced by the first author \cite{c2}, generalizing previous approach to K\"ahler-Einstein metrics. Twisted csck metrics appear in the work of J. Fine \cite{Fine}, Song-Tian \cite{ST}, Stoppa \cite{Stoppa1} and Darvas-Berman-Lu \cite{bdl} to highlight a few works in a very fast expanding literature. The main advantage of Chen's continuity path is that a twisted cscK metric is a minimizer of twisted K-energy, which is always  strictly convex. In particular this implies that the kernel of the linearized operator (of fourth order) is always zero, hence openness holds. For more details and discussion, see Chen \cite{c2}. More precisely, for $t \in [0,1]$ we consider the following family of equations: 
\begin{equation}\label{CSCK}
t(\underline R-R_\psi)+(1-t)(\tr_{\omega_\psi}\omega-n)=0, \ \ \psi \in \mathcal H.
\end{equation}
For $t=1$ this equation reduces to the cscK equation. By the results of \cite{Zeng} and \cite{Hashimoto}, this equation is always solvable in a neigborhood of $t=0$. As alluded to above, the values $t \in [0,1]$ for which the above equation is solvable forms an open set \cite{c2}. If the equation has a smooth solution up to $t=1$, then a desired cscK metric exists. Hence we can assume that there exists a maximal interval of the type $[0, T)$, for which solutions exist, and want to study what happens at the maximal singular time $T$. In this direction we note the following result:

\begin{theorem}\label{estimate}Suppose that the K-energy is proper on $\mathcal H$. If we assume that the Ricci curvature is bounded above along the continuity path then $T=1$ and a smooth cscK metric exists in $\mathcal H$.
\end{theorem}

As we will see, this theorem will be a direct consequence of Theorem \ref{main}(i) above and  \cite[Theorem 1.7]{HZ}. Ideally, we expect the assumption on the Ricci curvature bound to be a technicality. This result suggests that if we assume  properness of the K-energy, then any upper bound for the Ricci curvature has to blow up as we are near singular time. %the main point is that we do not assume Ricci curvature lower bound, as this can be obtained through the equation.

Finally, let us summarize by comparing our findings with known results in K\"ahler-Einstein and complex Monge-Amp\`ere theory. Our results giving $C^2$ bounds rely crucially on the Ricci upper bound. As a comparison, for the complex Monge-Amp\'ere equation, the K\"ahler-Einstein equation, or the K\"ahler-Ricci flow, one obtains the second order estimates using the $C^0$ estimate (the $C^0$ estimates in turn can be obtained through the equation directly, or  using properness). Such a second order estimate is not known for fourth order equations like the cscK equation, or the Calabi flow. This is a crucial technical difficulty, and represents one of the major technical differences between second order equations and fourth order equations in K\"ahler geometry. Our results suggest that the second order estimates and the Ricci upper bound are essentially equivalent for the cscK equation, assuming properness. 
However it seems to be an extremely hard problem in general to obtain a Ricci upper bound, or even a Laplacian bound, along the Calabi flow or along the continuity path.

\section{The compactness theorem}

In this section we will prove Theorem 1.2. For this we need to use several delicate theorems from pluripotential theory that we now recall:

\begin{theorem} \label{Skoda} Suppose $\mathcal L \subset \textup{PSH}(X,\o)$ is $L^1$ weak compact  and the Lelong numbers of all elements in $\mathcal L$ are zero. Then for any $\alpha > 0$ there exists $C(\alpha,\o)>0$ such that
$$\int_X e^{-\alpha u} \o^n \leq C, \ u \in \mathcal L.$$
\end{theorem}

This theorem is a well known consequence of Skoda's uniform integrability theorem, as proved in \cite[Corollary 3.2]{ze} and will be one of the key ingredients in our argument. As courtesy to the reader, we give a proof.

\begin{proof} Let $ U_1, \ldots, U_k, V_1, \ldots, V_k \subset X$ be coordinate neighborhoods such that $\overline V_k \subset U_k$, $ V_1, \ldots, V_k$ covers $X$ and  $i\partial\bar \partial \beta_j = \o$ for some $\beta_j \in C^\infty(U_j), j =1,\ldots, k$. We introduce the following local families:
$$\mathcal L_j = \{ \beta_j + v \ | \ v \in \mathcal L \} \subset \textup{PSH}(U_j).$$
As $\mathcal L \subset \textup{PSH}(X,\o)$ is weak $L^1$ compact, for any $\alpha >0$ the families $\alpha \mathcal L_j \subset \textup{PSH}(U_j)$ are also  weak $L^1$ compact. For analogous reasons, the elements of  $\alpha \mathcal L_j$ have zero Lelong numbers, hence we can apply  \cite[Corollary 3.2]{ze} which yields:
$$\int_{V_j} e^{-\alpha u} \o^n \leq C_j(\alpha,\mathcal L_j), \ u \in \mathcal L_j, j =1,\ldots, k.$$
Summing up over the coordinate patches yields the desired estimate.
\end{proof}

\begin{theorem}[Kolodziej's $C^0$ estimate]\label{Kolotheorem} Suppose $f \in L^1(X)$ is such that $f \geq 0$, $\int_X f \o^n =1$ and

$$\int_X  f \log(1+f)^p \o^n < C$$
for some $p > n$ and $C <\infty$. Then there exists $u \in PSH(X,\o)\cap L^\infty$ with $\sup_X u =0$ and $\sup_X |u| \leq D(C,p,\o)$ such that
$$(\o + i\partial\bar \partial u)^n = f\o^n.$$
\end{theorem}
The $d_p$ metric on $\mathcal H$ is the metric induced by the $L^p$ Mabuchi geometry on $\mathcal H$, as in studied in \cite{da}. For a survey on these metrics and related matters, we refer to \cite{da2}. The only things of importance here are that the every $d_p$ metric dominates the $d_1$ metric which in turn dominates the weak $L^1$ topology of $\textup{PSH}(X,\o)$. There is also the following double estimate:
\begin{equation}\label{doubleest}
\frac{1}{C}d_p(u,v) \leq \Big(\int_X |u-v|^p\o_u^n\Big)^{1/p} + \Big(\int_X |u-v|^p\o_v^n\Big)^{1/p} \leq C d_p(u,v).
\end{equation}
This last estimate implies that Aubin's $\mathcal I$ functional, recalled below, is dominated by all metrics $d_p$.
\begin{equation}\label{AubinIest}
\mathcal I(\o_u, \o_v)=\int_X (v-u) (\omega_u^n-\omega_v^n) \leq C d_p(u,v), \ u,v \in \mathcal H.
\end{equation}
Lastly, we record the following compactness theorem, which is a consequence of Theorem \ref{Skoda} and \eqref{doubleest}:
\begin{theorem}[strong compactness]\textup{\cite[Proposition 2.6, Theorem 2.17]{bbegz}} \label{bbegzcomp} Suppose $\{ u_k\}_{k \in \Bbb N} \subset \mathcal H$ is such that $|\sup_X u_k| \leq D$ and $\int_X \log \frac{\o_{u_k}^n}{\o^n} \o^n_{u_k} \leq D$ for some $D \geq0$. Then there exists $u \in \mathcal E^1(X,\o)$ and $k_l \to \infty$ such that $\int_X |u_{k_l}- u| \o^n \to 0$.
\end{theorem}

Actually, we also have $\lim_{l \to \infty} d_1(u_{k_l},u)=0$ in the above theorem, but this will not be important for us. The crucial fact here is that the elements of $\mathcal E^1(X,\o)$ have zero Lelong numbers \cite[Corollary 1.8]{gz}. Now we are ready to prove our main result in this section. 

\begin{theorem}\label{main}Let $\mathcal L \subset \mathcal H$ for which there exists $C>0$ satisfying:
$$\mathcal I(\o, \o_\phi) \leq C, \ Ric \o_\phi \leq C \o_\phi, \ \phi \in \mathcal L.$$
Then there exists $C'(\mathcal L) >0$ such that:
$$0 \leq \o +i\partial\bar\partial \phi \leq C' \o, \ \phi \in \mathcal L.$$
\end{theorem}
\begin{proof} Without loss of generality we can assume that $\int_X \o^n =1$.
First we establish the $C^0$ bound. We can suppose that $\sup_X \phi =0$ for all $\phi \in \mathcal L$. This implies that $\int_X \phi \o^n$ is uniformly bounded, hence by the $\mathcal I$ functional bound, also $\int_X \phi \o_{\phi}^n$ is uniformly bounded.

Let $F_\phi= \log(\o_\phi^n/\o^n)$ for $\phi \in \mathcal L$. We have
\begin{equation}\label{secondderivest}
Ric \ \o_\phi -Ric \ \o = -\i \partial \bar \partial F_\phi.
\end{equation}
We can suppose that $-Ric \ \o \leq C\o$ and $Ric \ \o_\phi \leq C \o_\phi$. Hence we have \[C(\o + \o_\phi)\geq -i\partial\bar\partial F_\phi.\] 
It follows that $\frac{1}{2}\phi + \frac{1}{2C}F_\phi\in \mathcal H$.  By Jensen's inequality, we also have that
$$ \int_X F_\phi \o^n \leq \log \int_X \frac{\o_u^n}{\o^n} \o^n=0.$$
Putting these facts together we obtain that there exists $D>0$ such that $ \sup_X (\frac{1}{2}\phi + \frac{1}{2C}F_\phi) \approx \int_X (\frac{1}{2}\phi + \frac{1}{2C}F_\phi)\o^n  \leq 0$, hence
\begin{equation}\label{trivest}F_\phi+C\phi  \leq D,
\end{equation}
for some $D>0$, which in turn implies that
\begin{equation}\label{trivexpest}
\o_\phi^n \leq D' e^{-C\phi}\o^n.
\end{equation}
%We also note that for positive constant $D>0$:
%\begin{equation}
%\Delta (F_\phi+C\phi)\geq -D.
%\end{equation}
Next we claim that the weak $L^1$ closure of $\mathcal L$ is a compact family (in the weak $L^1$ topology of $\textup{PSH}(X,\o)$) with zero Lelong numbers. Compactness is guaranteed by the fact that $\sup_X \phi =0$ for all $\phi \in \mathcal L$. To verify the condition on zero Lelong numbers, let $\phi_k \in \mathcal L$ such that $\phi_k \to_{L^1} \phi \in \overline {\mathcal L} \subset \textup{PSH}(X,\o)$. If we can argue that $\phi \in \mathcal E^1(X,\o)$, then we are done as elements of $\mathcal E^1(X,\o)$ have zero Lelong numbers \cite[Corollary 1.8]{gz}. But this follows as the conditions of Theorem \ref{bbegzcomp} are verified. Indeed, we have that
$$ \int_X \log \frac{\o_{\phi_k^n}}{\o^n} \o^n_{\phi_k} = \int_X F_{\phi_k} \o_{\phi_k}^n\leq \int_X (D/c - \phi_k/c) \o^n_{\phi_k} \leq E,$$
where the last estimate follows from our choice of normalization at the beginning of the proof. Hence, we can apply Theorem \ref{bbegzcomp} which gives the claim. Using the claim and Theorem \ref{Skoda} we conclude that for any $\alpha >0$ there exists $C(\alpha, \mathcal L) >0$ such that
\begin{equation}\label{Skodaest}
\int_X e^{-\alpha \phi} \o^n \leq C, \ \phi \in \mathcal L.
\end{equation}
For any $p \geq 1$, we can start to write:
\begin{flalign*}
\int_X \Big(\frac{\o^n_\phi}{\o^n}\Big)^p\o^n  \leq D \int_X e^{-Cp\phi}\o^n \leq  C(p,\mathcal L), \ \phi \in \mathcal L,
\end{flalign*}
where we have used \eqref{trivexpest} and  \eqref{Skodaest}. Finally, by choosing $p \geq 2$, we can apply Kolodziej's estimates (Theorem \ref{Kolotheorem}) to conclude the proof of the uniform $C^0$ estimate.

With the $C^0$ bound in hand, the bound on $\Delta_\o \phi$ is derived using Yau' techniques \cite{y}. Fix a large constant $C_3>0$, that will eventually be under control. At the point $p \in X$ where $\exp(-C_3\phi)(n+\Delta_\o \phi)$ is maximized we obtain:
\begin{eqnarray*}
\Delta_{\o_\phi}\left\{\exp(-C_3\phi)(n+\Delta
\phi)\right\}(p)\leq 0.\end{eqnarray*} 
Then, using normal coordinates, Yau's calculation yields:
\begin{eqnarray}\label{yau}0&\geq& \Delta
F_\phi-n^2\inf_{i\neq
l}R_{i\bar{i}l\bar{l}}-C_3n(n+\Delta_\o \phi)\nonumber\\
&&\quad+\left(C_3+\inf_{i\neq
l}R_{i\bar{i}l\bar{l}}\right)\exp\left\{\frac{-F}{n-1}\right\}(n+\Delta
\phi)^{n/(n-1)}.\end{eqnarray} 
Using \eqref{secondderivest}, we continue to write
\begin{eqnarray}0&\geq&
-C\Delta_\o \phi-C_2-n^2\inf_{i\neq
l}R_{i\bar{i}l\bar{l}}-C_3n(n+\Delta \phi)\nonumber\\
&&\quad+\left(C_3+\inf_{i\neq
l}R_{i\bar{i}l\bar{l}}\right)\exp\left\{\frac{-F}{n-1}\right\}(n+\Delta_\o
\phi)^{n/(n-1)}.\end{eqnarray} This implies that
$(n+\Delta_\o\phi)(p)$ has an upper bound $C_0$ depending only
on $\sup_X F$. It follows that
\begin{eqnarray*} \exp(-C_3\phi)(n+\Delta_\o \phi) &\leq&
\exp(-C_3\phi)(n+\Delta_\o \phi)(p)\\
&\leq& C_0\exp(-C_3\phi(p)).\end{eqnarray*} Because $\phi$
is uniformly bounded, we obtain \begin{equation}\label{5-5}
0<n+\Delta \phi \leq C(C_3, \sup_X F_\phi, X).\end{equation}
Hence $\Delta_\o \phi$ is uniformly bounded. %Using \eqref{secondderivest}, we conclude that $\sqrt{-1} \p\bar \p F_\phi\geq -C \omega$. 
\end{proof}
As the next example shows, it is not possible to obtain a lower bound for the metrics without further assumptions. Let $X$ be the torus $\Bbb C / (\Bbb Z + i \Bbb Z)$ with the flat metric $\o = i dz \wedge \bar{dz}$. For small enough $\varepsilon >0$ let $\alpha_\varepsilon: X \to \Bbb R$ be functions satisfying the following properties: $0 <  \alpha_\varepsilon <2$, $\int_X \alpha_\varepsilon \o=1$, $\alpha_\varepsilon\big|_{B(0,1/6)}=\varepsilon + |z|^2$  and $\alpha_\varepsilon\big|_{X\setminus B(0,1/4)}$ is independent of $\varepsilon$. Because $\dim(X) =1$, there exists $\beta_\varepsilon \in \mathcal H$ such that $\o_{\beta_\varepsilon} = \alpha_\varepsilon \o$. For the Ricci curvature of $\o_{\beta_\varepsilon}$ we have
$$Ric \ \o_{\beta_\varepsilon}\Big|_{B(0,1/6)} = - \i \partial \bar \partial \log \alpha_\varepsilon \Big|_{B(0,1/6)} = -i\frac{\varepsilon}{(\varepsilon + |z|^2)^2}dz \wedge \bar{dz} \leq 0.$$
Using this, one can see that $Ric \ \o_{\beta_\varepsilon}$ is in fact uniformly bounded above on $X$.
Hence, for small enough $\varepsilon$ we obtained a family of metrics $\{ \o_{\beta_\varepsilon}\}_{\varepsilon>0}$ such that $\mathcal I(\o,\o_{\beta_\varepsilon})$ and $Ric \ \o_{\beta_\varepsilon}$ are uniformly bounded above but $\o_{\beta_\varepsilon}$ is not uniformly bounded away from zero.

Combining the conclusion of our last result with the compactness theorem \cite[Theorem 5.1]{ch1}, we obtain the following corollary:
\begin{corollary}\label{cor1}Let $\mathcal L \subset \mathcal H$ for which there exists $C>0$ satisfying:
$$\mathcal I(\o, \o_\phi) \leq C, \  - C \o_\phi \leq Ric \ \o_\phi \leq C \o_\phi, \ \phi \in \mathcal L.$$
Then for any $\alpha \in [0,1)$ there exists $C'(\mathcal L) >1$ such that:
$$  \frac{1}{C'} \o \leq \o_\phi \leq C'\o, \  \| \phi\|_{C^{3,\alpha}} \leq C', \ \phi \in \mathcal L.$$
\end{corollary}
We note that this last result also improves \cite[Theorem 1]{sz}, where instead of bounded Ricci curvature the author assumes boundedness of the full curvature tensor, via a result of Schoen-Uhlenbeck \cite{su}.

%As a direct consequence of this compactness, we have the following result on the Calabi flow, which strengthens \cite[Theorem 2]{sz}.
%\begin{theorem}Suppose $[0,\infty) \ni  t \to c_t \in \mathcal H$ is a Calabi flow trajectory along which the Ricci curvature is bounded. Then the following hold:  If the Mabuchi K-energy is proper on $X$ then there exists a K\"ahler metric with constant scalar curvature in $\mathcal H$ and $t \to c_t$ converges smoothly to one such metric.

%\end{theorem}

%\begin{proof} 
% As $t \to c_t$ decreases the Mabuchi's K-energy, it follows that $\mathcal I(\o_{c_0},\o_{c_t})$ is bounded. Recall that the Aubin-Mabuchi energy is constant along any Calabi flow trajectory, hence $AM(c_t)=AM(c_0)$. As argued in the last section of \cite{da}, together with the boundedness of $I(\o_{c_0,\o_{c_t}})$ implies that $|\sup_X c_t|$ is also bounded. We can apply now the results of Chen-He to conclude the argument.
%\end{proof}

\begin{rmk}
It would be very interesting to understand the compactness of potentials with weaker curvature conditions, such as replacing Ricci curvature by scalar curvature. We ask: is a family of K\"ahler metrics in a fixed K\"ahler class with uniformly bounded scalar curvature and uniformly bounded potential compact (say in $C^{1, \alpha}$ topology)? Namely, can we obtain second order estimates on the potential assuming scalar curvature bound and $C^0$  bound? 
\end{rmk}

\section{Applications to the Calabi flow}

We study applications to the Calabi flow in this section. Recall that in the presence of bounded Ricci curvature, the first and third named authors proved the long time existence of the Calabi flow.

\begin{proof}[Proof of Theorem \ref{thm: Calabif}] First we prove part (i). If $\phi \in \mathcal H$ is cscK then by the distance shrinking property of the Calabi flow and \eqref{AubinIest} we obtain that $$\mathcal I(\o_\phi,\o_{c_t}) \leq Cd_2(\phi,c_t)\leq Cd_2(\phi,c_0), \ t \geq 0.$$
By \cite[Corollary 4]{da} we have additionally that $ |\sup_X c_t| $ is bounded by $d_2(\phi,c_t)$. Given all this and the Ricci curvature bound, we can apply the compactness theorem of the previous section to conclude that $\{c_t \}_t$ is $C^{3,\alpha}$-compact. Hence there exists a $C^{3, \alpha}$ K\"ahler potential $c_\infty$ that minimizes the K-energy. Using \cite[Theorem 1.7]{HZ} (or more generally \cite[Theorem 1.1]{bdl}), we conclude that $c_\infty$ is in fact a smooth cscK potential.

Part (ii) is argued in a similar way. Using the formalism of \cite{DR}, properness of the K-energy simply means that $\mathcal K(u) \geq C d_1(0,u) - D$ for all $u \in \mathcal H$ and some $C,D>0$ (See \cite[Proposition 5.5]{DR}). 

 As $t \to c_t$ decreases the Mabuchi K-energy, it follows that $d_1(0,{c_t})$ is bounded. Using again \cite[Corollary 4]{da} we get that $|\sup_X c_t|$ is bounded. We can now apply our compactness result to conclude the argument, as in part (i).
\end{proof}

\section{Applications to the method of continuity}

In this section, we denote by $\psi_t, t \in [0,1]$ the potential solutions to the twisted cscK equation along the continuity path \eqref{CSCK}. It is well known that such $\psi_t$ minimize the twisted K-energy:
$$\mathcal K_t(u)=\mathcal K(u) + \frac{1-t}{t}\mathcal J(u), \ u \in \mathcal H,$$
where $\mathcal K$ is Mabuchi's K-energy and $\mathcal J$ is the following functional:
$$\mathcal J(u) = \frac{1}{V }\sum_{j=0}^{n-1}\int_X u \o^j \wedge \o_u^{n-1-j} - \frac{n}{(n+1) V }\sum_{j=0}^{n} \int_X u \o^j \wedge \o_u^{n-j},  \ \ u \in \mathcal H.$$

Finally, we give the argument of our last main result:

\begin{proof}[Proof of Theorem \ref{estimate}] 
Since $\mathcal J \geq 0$ it follows that $\mathcal K_{t_1} \geq \mathcal K_{t_2}$ for any $t_1,t_2 \in (0,T)$ with $t_1 \leq t_2$. In particular, since $\psi_t$ minimizes $\mathcal K_t$, it follows that $t \to \mathcal K_{t}(\psi_t)$ is decreasing for $t \in [0,T)$. We know that $\mathcal K$ is proper, i.e., $\mathcal K(\cdot) \geq C \mathcal J(\cdot) - D$ for some $C,D>0$. Putting everything together we obtain that $\mathcal J(\psi_t)$ is bounded for $t \in [0,T)$, hence so is $\mathcal I(\o,\o_{\psi_t})$.

By Theorem \ref{main}(i),  $\Delta_\o \psi_t$ has to be uniformly bounded along the continuity path. On the other hand, using Theorem \cite[Theorem 1.7]{HZ} we obtain that under such circumstance $T=1$ and a smooth cscK metric exists.
\end{proof}

\paragraph{Acknowledgements.} The  first named author has been partially supported by NSF grant DMS--1515795. The second named author has been partially supported by NSF grant DMS--1610202 and BSF grant 2012236. The third named author has been partially supported by NSF grant DMS--1611797.

\normalsize
\noindent{\sc Stony Brook University}\\
{\tt xiu@math.sunysb.edu}\vspace{0.1in}\\
\noindent{\sc University of Maryland}\\
{\tt tdarvas@math.umd.edu}\vspace{0.1in}\\
\noindent {\sc University of Oregon}\\
{\tt whe@uoregon.edu }
\end{document}